\documentclass[11pt,a4paper]{article}
\usepackage[utf8]{inputenc}
\usepackage[T1]{fontenc}
\usepackage{amsmath,amssymb,amsthm,mathrsfs}
\usepackage[margin=2.55cm]{geometry}
\usepackage[svgnames]{xcolor}
\usepackage{graphicx}
\usepackage{enumitem}
\usepackage[colorlinks=true,linkcolor=NavyBlue,citecolor=NavyBlue,urlcolor=NavyBlue]{hyperref}

\definecolor{melicolor}{RGB}{186,85,211}   
\definecolor{danicolor}{RGB}{86,185,251}   
\definecolor{tomascolor}{RGB}{186,120,25}   

\newtheorem{theorem}{Theorem}[section]
\newtheorem{lemma}[theorem]{Lemma}
\newtheorem{corollary}[theorem]{Corollary}
\newtheorem{proposition}[theorem]{Proposition}
\theoremstyle{definition}
\newtheorem{definition}[theorem]{Definition}
\newtheorem{example}[theorem]{Example}
\newtheorem{remark}[theorem]{Remark}

\newcommand{\N}{\mathbb{N}}
\newcommand{\Z}{\mathbb{Z}}
\newcommand{\R}{\mathbb{R}}
\newcommand{\C}{\mathbb{C}}
\newcommand{\D}{\mathbb{D}}
\newcommand{\T}{\mathbb{T}}
\newcommand{\Dinf}{\mathbb{D}^\infty_2}
\newcommand{\Dcl}{\overline{\mathbb{D}}^{\,\infty}}
\newcommand{\Tinf}{\mathbb{T}^\infty}

\newcommand{\cH}{\mathcal{H}}
\newcommand{\cT}{\mathcal{T}}
\newcommand{\cB}{\mathfrak{B}}
\newcommand{\HN}{H_2(\Dinf,\C^2)}
\newcommand{\pint}[1]{\left\langle #1\right\rangle}
\newcommand{\spanc}{\overline{\operatorname{span}}}
\newcommand{\smin}{\sigma_{\min}}
\newcommand{\smax}{\sigma_{\max}}
\newcommand{\MN}{M_2(\C)}
\renewcommand{\Re}{\operatorname{Re}}
\DeclareMathOperator{\adj}{adj}

\title{Dilation systems with two generators\\ and the optimal Riesz bounds of a basis of ReLU networks}
\author{Jorge Antezana\thanks{Supported by PID2024-160033NB-I00 funded by MICIU/AEI/10.13039/501100011033 and by FEDER, UE.}
\and Daniel Carando, \thanks{Supported by CONICET PIP 11220250100329CO and UBACYT 20020250200096BA.}
\and Tom\'as Fern\'andez Vidal\thanks{Supported by CONICET PIP 11220250100329CO and UBACYT 20020250200096BA.}
\and Melisa Scotti\thanks{Supported by CONICET PIP 11220250100329CO and UBACYT 20020250200096BA.} }

\begin{document}
\maketitle

\begin{abstract}
For two $1$-periodic functions $\psi_1,\psi_2\in L^2(0,1)$ with mean zero we characterize when $\{1\}\cup\{\psi_1(\alpha\cdot x),\psi_2(\alpha\cdot x)\}_{\alpha\in\Z^d,\,\alpha\succ0}$ is a Riesz basis of $L^2([0,1]^d)$, in terms of a $2\times2$ matrix of Dirichlet series, and we compute its optimal bounds. The tool is a $2\times2$ matrix version of the theorem of Hedenmalm, Lindqvist and Seip for systems of dilations with two generators, proved through the reproducing kernel of the Hardy space of the infinite polydisk. In particular, we get the optimal bounds $2/3$ and $3/2$ for $\psi_1,\psi_2$ the triangle waves considered by Daubechies, DeVore, Foucart, Hanin and Petrova in $L^2(0,1)$, and the extended version in $L^2([0,1]^d)$ defined by Schneider and Vyb\'\i ral.
\end{abstract}



\section{Introduction}\label{sec intro}

The real trigonometric system $\{1\}\cup\{\sqrt2\cos(2\pi nx),\sqrt2\sin(2\pi nx)\}_{n\in\N}$ is an orthonormal basis of $L^2(0,1)$. Daubechies, DeVore, Foucart, Hanin and Petrova \cite{DDFHP2022} observed that the cosine and the sine can be replaced by two $1$-periodic triangle waves $\mathcal C$ and $\mathcal S$ (see \eqref{eq triangle def} and Figure~\ref{fig profiles}). These are continuous and piecewise linear, hence exactly realizable by ReLU neural networks, and yet $\{1\}\cup\{\sqrt3\,\mathcal C(nx),\sqrt3\,\mathcal S(nx)\}_{n\in\N}$ is still a Riesz basis of $L^2(0,1)$. Schneider and Vyb\'\i ral \cite{SchneiderVybiral2024} introduced the multivariate system
\[
\mathfrak R_d=\{1\}\cup\big\{\sqrt3\,\mathcal C(\alpha\cdot x),\ \sqrt3\,\mathcal S(\alpha\cdot x):\ \alpha\in\Z^d,\ \alpha\succ0\big\},
\]
where $\alpha\succ0$ means that the first nonzero coordinate of $\alpha$ is positive. They proved that it is a Riesz basis of $L^2([0,1]^d)$ with bounds $1/2$ and $3/2$ for every $d$, and asked whether these constants are optimal. Kulbatov, Lang, Schneider and Vyb\'\i ral \cite{KLSV2025} improved the lower bound to $0.5787\ldots$ and showed that $\mathfrak R_d$ is also a Schauder basis of $L_q([0,1]^d)$, $1<q<\infty$. Our first result, Theorem~\ref{thm relu}, answers the optimality question: for every $d\ge1$ the optimal Riesz bounds of $\mathfrak R_d$ in $L^2([0,1]^d)$ are $A=2/3$ and $B=3/2$ (in the recent article \cite{BillnerSchneider2026} these bounds were obtained independently, see comments in next paragraph). To get the lower bound, we use that the Fourier coefficients of $\sqrt3\,\mathcal C$ with respect to $\sqrt2\cos(2\pi nt)$ are $\mu n^{-2}$ for odd $n$ and $0$ for even $n$, with $\mu=4\sqrt6/\pi^2$. Then,  following Hedenmalm, Lindqvist and Seip \cite{HeLiSe}, we consider the corresponding Dirichlet series $D_{\mathcal C}(s)=\mu\sum_{n\ \rm odd}n^{-2-s}=\mu\prod_{p>2}(1-p^{-2}p^{-s})^{-1}$, which is an Euler product. The optimal Riesz bounds of $\mathfrak R_d$ turn out to be the squared infimum and supremum of $|D_{\mathcal C}|$, and of the analogous series $|D_{\mathcal S}|$, on the half-plane $\Re s>0$. From Bohr's point of view each $p^{-s}$ is an independent variable in the unit disk, so these extremes are the products of the extremes of the factors, $\mu\prod_{p>2}(1+p^{-2})^{-1}=\sqrt{2/3}$ and $\mu\prod_{p>2}(1-p^{-2})^{-1}=\sqrt{3/2}$.

While this paper was being completed, Billner and Schneider \cite{BillnerSchneider2026} obtained the lower bound $2/3$ for $\mathfrak R_d$, in every dimension, by a different method. 
For a single sequence $\tau$ generating both the cosine and the sine part, they also characterize the Riesz bases of this type, with optimal constants, through the multipliers of the Hardy space of Dirichlet series, and they treat mixed systems whose two parts come from different sequences. In the language of this paper, the $2\times2$ matrix associated with the pair of generators (the symbol) is diagonal.

Our main result, Theorem~\ref{thm two profiles},  is a characterization for two arbitrary generators . Given two $1$-periodic functions $\psi_1,\psi_2\in L^2(0,1)$ with mean zero, consider
\[
\mathfrak R_d(\psi_1,\psi_2)=\{1\}\cup\big\{\psi_1(\alpha\cdot x),\ \psi_2(\alpha\cdot x):\ \alpha\in\Z^d,\ \alpha\succ0\big\}.
\]
If we write the Fourier coefficients of $\psi_i$ with respect to $\sqrt2\cos(2\pi nt)$ and $\sqrt2\sin(2\pi nt)$ as $a^{(i)}_n$ and $b^{(i)}_n$, let $D_{a^{(i)}}(s)=\sum_na^{(i)}_nn^{-s}$ and $D_{b^{(i)}}(s)=\sum_nb^{(i)}_nn^{-s}$. We also let $\mathcal H_\infty$ be the algebra of Dirichlet series that converge to bounded functions in the half-plane $\Re s>0$. Theorem~\ref{thm two profiles}, in the form of Remark \ref{rem del teo}, says that, for every $d\ge1$, $\mathfrak R_d(\psi_1,\psi_2)$ is a Riesz basis of $L^2([0,1]^d)$ if and only if the four series $D_{a^{(i)}},D_{b^{(i)}}$ belong to $\mathcal H_\infty$ and
\[
\frac{1}{D_{a^{(1)}}D_{b^{(2)}}-D_{a^{(2)}}D_{b^{(1)}}}\in\mathcal H_\infty.
\]
Moreover, the optimal Riesz bounds are then $\min\{1,A_\Theta\}$ and $\max\{1,B_\Theta\}$, where $A_\Theta$ and $B_\Theta$ are the extreme squared singular values, over the infinite polydisk, of the $2\times2$ matrix of Bohr lifts of the four series. In particular, they do not depend on $d$. For the triangle waves the matrix is diagonal, with entries the Bohr lifts of $D_{\mathcal C}$ and $D_{\mathcal S}$, and Theorem~\ref{thm relu} is the resulting computation. Theorem~\ref{thm two profiles} also covers a single generator dilated along all $\alpha\neq0$ (Corollary~\ref{cor general profile}, Example~\ref{ex single profile}), and in Example~\ref{ex nondiagonal} we work out a pair with a triangular, non-diagonal matrix, for which the bounds come from $2\times2$ singular values and are again explicit.

Let us describe the method (see section \ref{subsec notation} for notation). For $\varphi\in L^2(0,1)$, extended to $\R$ as an odd $2$-periodic function $\varphi(x)=\sum_na_n\sqrt2\sin(\pi nx)$, Hedenmalm, Lindqvist and Seip \cite{HeLiSe} proved that $\{\varphi(n\,\cdot)\}_{n\in\N}$ is a Riesz basis of $L^2(0,1)$ if and only if $D_\varphi(s)=\sum_na_nn^{-s}$ and $1/D_\varphi$ belong to $\mathcal H_\infty$ (see \cite{ACS_splitting} for  a study of  Bessel and Riesz sequence properties separately, and \cite{Nikolski2012,DanGuo2021} for the question on completeness). Under the Bohr lift $n^{-s}\mapsto z^{\alpha(n)}$, with $\alpha(n)$ the exponents in the prime factorization of $n$, $D_\varphi$ becomes a holomorphic function on the infinite polydisk and the dilation by $k$ becomes multiplication by $z^{\alpha(k)}$. The systems $\mathfrak R_d(\psi_1,\psi_2)$ have, in each primitive direction $\beta\in\Z^d$, two generators $\psi_1(\beta\cdot x)$, $\psi_2(\beta\cdot x)$ and two orthogonal ``frequency lines'', spanned by the dilations of $\cos(2\pi\beta\cdot x)$ and of $\sin(2\pi\beta\cdot x)$. We therefore prove a version of the Hedenmalm--Lindqvist--Seip theorem for two generators on two frequency lines. Here,  $D_\varphi$ is replaced by the $2\times2$ matrix of Dirichlet series formed with the coefficients of each generator along each line (Proposition~\ref{prop matrix}): the system is a Riesz basis if and only if the entries and the  inverse of the determinant belong to $\mathcal H_\infty$. Moreover, its optimal bounds are the extreme squared singular values of the Bohr lift of the matrix. Theorem~\ref{thm two profiles} follows by applying this in each primitive direction. The idea of the proof is the following: after the Bohr lift, the synthesis operator of the system becomes multiplication by the matrix on the $\C^2$-valued Hardy space of the infinite polydisk, and the reproducing kernel of that space turns the norms of the operator and of its inverse into pointwise norms of the matrix and of its inverse. Finally, the identification of $\mathcal H_\infty$ with the algebra of bounded holomorphic functions on the polydisk \cite{DirSer,HeLiSe} translates these pointwise conditions on the matrix into conditions on the Dirichlet series.

In  Section~\ref{sec results} we state our  main results and  present different examples.  Section~\ref{sec proofs} contains all the proofs.



\subsection{Notation and conventions}\label{subsec notation}

\paragraph{Dirichlet series and the Bohr lift.} A Dirichlet series is a formal expression of the form  $$\sum_{n=1}^\infty a_n n^{-s},$$ where the coefficients $a_n$ are complex numbers and $s$ is a complex variable. The Hardy space $\mathcal{H}_2$ consists of those Dirichlet series whose coefficients sequence $(a_n)_n$ belong to $\ell_2$, with norm $$\Vert D\Vert_{\mathcal{H}_2} = \Big(\sum_{n=1}^\infty |a_n|^2\Big)^{1/2}.$$

In \cite{HeLiSe}, Hedenmalm, Lindqvist, and Seip characterized the multipliers of $\mathcal{H}_2$. That is, the functions $\varphi$ for which $\varphi D \in \mathcal{H}_2$ for every $D\in \mathcal{H}_2$. Since every Dirichlet series in $\mathcal{H}_2$ converges in the half-plane $\mathbb{C}_{1/2} :=\{s\in \mathbb{C}: \mbox{Re}(s)>1/2\},$
 and $D = 1$ belongs to $\mathcal{H}_2$, it follows that every multiplier must  be a Dirichlet series convergent on $\mathbb{C}_{1/2}.$  The authors proved that if $D$ is a multiplier, then $D$ actually converges in the whole right half-plane $\mathbb{C}_0:=\{s\in \mathbb{C}: \mbox{Re}(s)>0\}$, and defines there a bounded holomorphic function. Consequently, the multiplier space of $\mathcal{H}_2$ is
$$\mathcal{H}_\infty :=\Big\{ D(s)=\sum_{n=1}^\infty a_n n^{-s} :\; \Vert D\Vert_{\mathcal{H}_\infty}=\sup\limits_{\text{Re}(s)>0}\vert D(s) \vert <\infty\Big\}.$$

This space of multipliers is isometrically isomorphic to the Hardy spaces on the infinite polytorus $\mathbb{T}^\infty$, the compact group endowed with the product topology, defined by
$$H_\infty(\mathbb{T}^\infty) = \Big\{ f\in L_\infty(\mathbb{T}^\infty,m) :  \hat{f}(\alpha )=0 \,\mbox{ if } \alpha\not\in \mathbb{N}_0^{(\mathbb{N})}\Big\},$$
with $m$ the Haar measure, $\mathbb{Z}^{(\mathbb{N})}$ the set of eventually null sequences of integer numbers, and for $w = (w_1,w_2,\ldots)\in \mathbb{T}^\infty$ and $\alpha\in \mathbb{Z}^{(\mathbb{N})}$ we write $$w^\alpha=w_1^{\alpha_1}\cdots w_n^{\alpha_n} \cdots $$
If $f\in L_1(\mathbb{T}^\infty )$, $\hat{f}(\alpha)$ denotes the $\alpha$-th Fourier coefficient given by 
$$\hat{f}(\alpha)=\int_{\mathbb{T}^\infty} f(w)w^{-\alpha} dm (w).$$

The isomorphism is given by the Bohr transform $\mathcal{B}$, which connects formal power and Dirichlet series. Given $n\in \mathbb{N}$, there exists a unique $\alpha\in \mathbb{N}_0^{(\mathbb{N})}$ such that $n = \mathfrak{p}^\alpha$, being $\mathfrak{p}$ the sequence of prime numbers and $\mathfrak{p}^\alpha = p_1^{\alpha_1}\cdot p_2^{\alpha_2} \cdots$. Then \begin{equation}\label{Def Bohr}
\mathcal{B}\Big(\sum\limits_{\alpha\in \mathbb{N}_0^{(\mathbb{N})}} a_\alpha z^\alpha\Big) = \sum\limits_{n=1}^\infty a_n n^{-s},
\end{equation}
where $a_n = a_\alpha$ if $n=\mathfrak{p}^\alpha$, defines a bijection between the formal power and Dirichlet series, and an isomorphism between $\mathcal{H}_\infty$ and $H_\infty(\mathbb{T}^\infty)$.

We call the inverse of $\mathcal B$ the \emph{Bohr lift}: the Bohr lift of a Dirichlet series $D(s)=\sum_n a_nn^{-s}$ is the power series
\[
F(z)=\sum_{n=1}^\infty a_n\,z^{\alpha(n)},
\]
obtained by replacing each $p_i^{-s}$ by an independent variable $z_i$, with $\alpha(n)$ such that $n=\mathfrak{p}^{\alpha(n)}$. Since $\alpha(mn)=\alpha(m)+\alpha(n)$, the Bohr lift is multiplicative: the Bohr lift of a product of Dirichlet series is the product of their Bohr lifts. For instance, the Bohr lift of an Euler product $\prod_{i}(1-c_ip_i^{-s})^{-1}$ is $\prod_i(1-c_iz_i)^{-1}$.

The power series $\sum_{\alpha\in \mathbb{N}_0^{(\mathbb{N})}} \hat{f}(\alpha) w^\alpha$, associated with a $f\in H_\infty(\mathbb{T}^\infty)$, converges absolutely on the domain in $\ell_2$ given by
$$\mathbb{D}_2^\infty : = \{z\in \ell_2 : \vert z_n \vert <1 \mbox{ for all } n\in \mathbb{N}\},$$
and thus, defines a holomorphic function there (see, for example, \cite[Theorem~13.2]{DirSer}). Therefore, given $f\in H_\infty(\mathbb{T}^\infty)$, we can define the holomorphic function $F(z) = \sum_{\alpha\in \mathbb{N}_0^{(\mathbb{N})}} \hat{f}(\alpha) z^\alpha$. Then,  $\hat{f}(\alpha)$  coincides with the  $\alpha$-th Taylor coefficient of $F$, and moreover,
\begin{equation}\label{def norma Hp escalar} 
\Vert f \Vert_{H_\infty(\mathbb{T}^\infty)} = \sup_{z\in \mathbb{D}_2^\infty} \vert F(z)\vert.
\end{equation}
The  Hardy space of holomorphic functions on $\mathbb{D}_2^\infty  $ is defined as
$$H_\infty(\mathbb{D}_2^\infty) :=\Big\{F:\mathbb{D}_2^\infty \to \mathbb{C} \mbox{ holomorphic : such that the supremum in } \eqref{def norma Hp escalar} \mbox{ is finite}\Big\}.$$
This space is isometrically isomorphic to $H_\infty(\mathbb{T}^\infty)$ and the isomorphism maps power series to Fourier series with the same coefficients.

We denote by $H_2(\Dinf)$ the Hardy space of functions $F(z)=\sum_{\alpha\in\mathbb{N}_0^{(\mathbb{N})}}c_\alpha z^\alpha$ on $\Dinf$ with $\|F\|^2=\sum_\alpha|c_\alpha|^2<\infty$, and by $H_2(\mathbb{T}^\infty)$ the subspace of those $f\in L_2(\mathbb{T}^\infty,m)$ such that $\hat f(\alpha)=0$ for every $\alpha\notin\mathbb{N}_0^{(\mathbb{N})}$. Both spaces are isometrically isomorphic to $\mathcal{H}_2$ via the Bohr transform, since in all three spaces the norm is the $\ell_2$ norm of the coefficients.

In \cite{ACFS_multipliers}, we extend the study of multipliers to the setting of Dirichlet series with coefficients in a Banach space. We refer to \cite{DirSer} for the general theory of Dirichlet series and their connection with infinite dimensional Fourier and Complex analysis.

\paragraph{Riesz bases.} A family $\{f_i\}_{i\in I}$ in a Hilbert space is a \emph{Riesz basis with bounds $A,B$} ($0<A\le B$) if it is complete and
\[
A\sum|c_i|^2\le\Big\|\sum c_if_i\Big\|^2\le B\sum|c_i|^2
\]
for all finitely supported $(c_i)$. Equivalently, $\{f_i\}$ is the image of an orthonormal basis $\{e_i\}$ under a bounded invertible operator $\cT$, with $\|\cT\|^2\le B$ and $\|\cT^{-1}\|^{-2}\ge A$ \cite[Section~3.6]{Christensen2016}. The optimal Riesz bounds are the largest $A$ and the smallest $B$ for which these inequalities hold. It  is well known that  they are $\|\cT^{-1}\|^{-2}$ and $\|\cT\|^2$. The upper inequality alone, with constant $B$, says that $\{f_i\}$ is a \emph{Bessel sequence} with bound $B$.

\section{Main results}\label{sec results}

\subsection{Triangle waves: optimal Riesz bounds}\label{subsec triangle}

Let $\mathcal C$ be the $1$-periodic function with
\begin{equation}\label{eq triangle def}
\mathcal C(t)=1-4|t|\quad\text{for }|t|\le\tfrac12,\qquad\text{and}\qquad \mathcal S(t)=\mathcal C(t-\tfrac14).
\end{equation}
On $[0,1]$ these are the functions $\mathcal C(x)=4|x-\tfrac12|-1$ and $\mathcal S(x)=|2-4|x-\tfrac14||-1$ of \cite{DDFHP2022,SchneiderVybiral2024}. Both are continuous and piecewise linear, hence realized exactly on any bounded interval by a one-hidden-layer ReLU network.

\begin{definition}
    Given $d\geq 1$, for $\alpha\in \mathbb{Z}^d$ we say that $\alpha\succ0$ if the first nonzero coordinate of $\alpha$ is positive.
\end{definition}

The following is our first main result, which we prove as a consequence of Theorem \ref{thm two profiles}.

\begin{theorem}\label{thm relu}
For every $d\ge1$ the system
\[
\mathfrak R_d=\{1\}\cup\big\{\sqrt3\,\mathcal C(\alpha\cdot x),\ \sqrt3\,\mathcal S(\alpha\cdot x):\ \alpha\in\Z^d,\ \alpha\succ0\big\}
\]
is a Riesz basis of $L^2([0,1]^d)$ with optimal Riesz bounds $A=2/3$ and $B=3/2$.
\end{theorem}

\subsection{Arbitrary pairs of generators}\label{subsec two profiles}

Let $\psi_1,\psi_2\in L^2(0,1)$ be $1$-periodic, with mean zero:
\[
\psi_i(t)=\sum_{n\ge1}\big(a^{(i)}_n\sqrt2\cos(2\pi nt)+b^{(i)}_n\sqrt2\sin(2\pi nt)\big).
\]
Our main examples are real-valued, but we allow complex coefficients, since the method requires it: the Bohr lift takes us to a Hardy space of holomorphic functions, which is a complex Hilbert space. This makes no difference for real generators, because the Riesz bounds in the real and in the complex space $L^2([0,1]^d)$ coincide: if the $\varphi_i$ are real and $c_i,d_i\in\R$, then $$\|\sum(c_i+id_i)\varphi_i\|^2=\|\sum c_i\varphi_i\|^2+\|\sum d_i\varphi_i\|^2,$$ which reduces the Riesz inequalities with complex coefficients to the real case, and $f\perp\varphi_i$ if and only if $\Re f\perp\varphi_i$ and $\operatorname{Im}f\perp\varphi_i$, which does the same for completeness.
We also assume that $\psi_1$ and $\psi_2$ have mean zero, and this is no loss of generality. Indeed, for every $\alpha\neq0$ the function $\psi_i(\alpha\cdot x)$ has the same mean as $\psi_i$, so if $\psi_i$ has mean $c\neq0$, then $\pint{1,\psi_i(\alpha\cdot x)}=c$ for every $\alpha\succ0$. Testing the Bessel inequality on the constant function $1$ then gives an infinite sum of terms equal to $|c|^2$, so $\{1\}\cup\{\psi_i(\alpha\cdot x)\}_{\alpha\succ0}$ is not even a Bessel sequence.

Our goal is to decide when the system $$\{1\}\cup\{\psi_1(\alpha\cdot x),\psi_2(\alpha\cdot x)\}_{\alpha\succ0}$$ is a Riesz basis of $L^2([0,1]^d)$. To this end we group the indices $\alpha$ by direction. Recall that $\alpha\succ0$ means that the first nonzero coordinate of $\alpha$ is positive. A vector $\beta\in\Z^d$ is \emph{primitive} if its coordinates have no common divisor $>1$. Every $\alpha\succ0$ can be written uniquely as $\alpha=n\beta$ with  $\beta\succ0$ primitive and  $n$  the greatest common divisor of the coordinates of $\alpha$. Thus the functions $\psi_i(\alpha\cdot x)$ with $\alpha$ in a fixed direction $\beta$ are the dilations $\psi_i(n\beta\cdot x)$ of the single function $\psi_i(\beta\cdot x)$.

Let $D_{a^{(i)}}(s)=\sum_na^{(i)}_nn^{-s}$ and $D_{b^{(i)}}(s)=\sum_nb^{(i)}_nn^{-s}$, let $F_{a^{(i)}},F_{b^{(i)}}\in H_2(\Dinf)$ be their Bohr lifts, and put
\[
\Theta_{\psi_1,\psi_2}=\begin{pmatrix}F_{a^{(1)}}&F_{a^{(2)}}\\ F_{b^{(1)}}&F_{b^{(2)}}\end{pmatrix}.
\]
We call $\Theta_{\psi_1,\psi_2}$ the \emph{symbol} of the system.

For a matrix $M$ we write $\smin(M)=\inf_{\|x\|=1}\|Mx\|$ and $\smax(M)=\sup_{\|x\|=1}\|Mx\|=\|M\|$, the singular values of $M$.
The following is our second main result and is proved in Section~\ref{subsec proof B}.

\begin{theorem}\label{thm two profiles}
For every $d\ge1$, the system
\[
\mathfrak R_d(\psi_1,\psi_2)=\{1\}\cup\big\{\psi_1(\alpha\cdot x),\ \psi_2(\alpha\cdot x):\ \alpha\in\Z^d,\ \alpha\succ0\big\}
\]
is a Riesz basis of $L^2([0,1]^d)$ if and only if the four entries of $\Theta_{\psi_1,\psi_2}$ are bounded on $\Dinf$ and $\inf_{z\in\Dinf}|\det\Theta_{\psi_1,\psi_2}(z)|>0$.
In that case the optimal Riesz bounds are
\[
A=\min\Big\{1,\ \inf_{z\in\Dinf}\smin\big(\Theta_{\psi_1,\psi_2}(z)\big)^2\Big\},\qquad
B=\max\Big\{1,\ \sup_{z\in\Dinf}\smax\big(\Theta_{\psi_1,\psi_2}(z)\big)^2\Big\}.
\]
\end{theorem}

\begin{remark}\label{rem del teo}
Note that the previous theorem  states that, in particular, the optimal Riesz bounds do not depend on $d$. 

Note also that if the four entries of $\Theta_{\psi_1,\psi_2}$ extend continuously to $\Dcl$ (the product of countably many copies of the closed unit disk, endowed with the product topology), the infimum and the supremum can  be taken over $\Tinf$, and they are attained there. If moreover $\mathfrak R_d(\psi_1,\psi_2)$ is a Riesz basis, then $\det\Theta_{\psi_1,\psi_2}$, being continuous and bounded away from $0$ on the dense set $\Dinf$, has no zeros on $\Dcl$, and the same holds for the infimum (see Lemma \ref{lem torus} below).

     In terms of Dirichlet series, Theorem \ref{thm two profiles} says that $\mathfrak{R}_d(\psi_1,\psi_2)$ is a Riesz basis of $L^2([0,1]^d)$ if and only if
     $D_{a^{(1)}}$, $D_{a^{(2)}}$, $D_{b^{(1)}}$, $D_{b^{(2)}}\in\mathcal H_\infty$ and
\[
\frac1{D_{a^{(1)}}D_{b^{(2)}}-D_{a^{(2)}}D_{b^{(1)}}}\in\mathcal H_\infty .
\]
\end{remark}

As a direct application of Theorem \ref{thm two profiles}, we present the proof of Theorem \ref{thm relu}. This corresponds to a  particular instance where the matrix $\Theta_{\psi_1,\psi_2}$  is diagonal. 
Recall the definition of the Dirichlet character $$\chi_{-4}(n)=\begin{cases}
    0 \; &\mbox{ if } n \mbox{ is even},\\
    (-1)^{(n-1)/2} \; &\mbox{ if } n \mbox{ is odd.}
\end{cases}$$
\begin{proof}[Proof of Theorem \ref{thm relu}]

For $\psi_1=\sqrt3\,\mathcal C$ and $\psi_2=\sqrt3\,\mathcal S$, their Fourier series are
\begin{equation}\label{eq triangle fourier}
\sqrt3\,\mathcal C(t)=\mu\sum_{m\ge0}\frac{\sqrt2\cos(2\pi(2m+1)t)}{(2m+1)^2},\qquad
\sqrt3\,\mathcal S(t)=\mu\sum_{m\ge0}(-1)^m\frac{\sqrt2\sin(2\pi(2m+1)t)}{(2m+1)^2},
\end{equation}
where $\mu=4\sqrt6/\pi^2$, and $\|\sqrt3\,\mathcal C\|_{L^2(0,1)}=\|\sqrt3\,\mathcal S\|_{L^2(0,1)}=1$.

\begin{figure}[t]
\centering
\includegraphics[width=\textwidth]{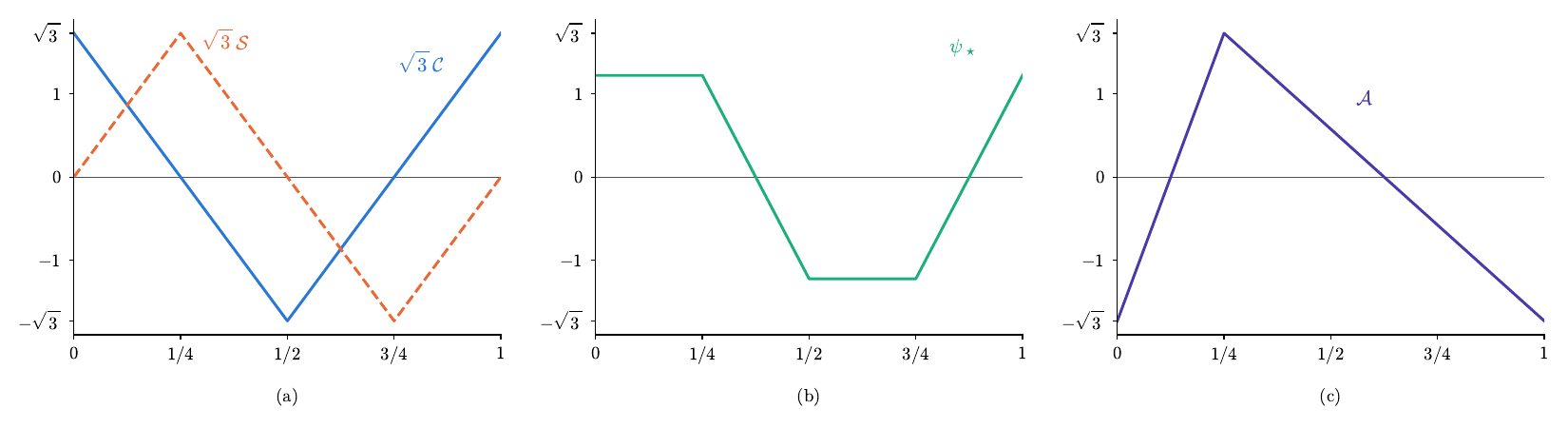}
\caption{(a) The triangle waves $\sqrt3\,\mathcal C$ and $\sqrt3\,\mathcal S$. (b) The trapezoidal wave $\psi_\star=\sqrt{3/2}\,(\mathcal C+\mathcal S)$ of Example~\ref{ex single profile}. (c) The asymmetric triangle wave $\mathcal A$ of Example~\ref{ex nondiagonal}.}
\label{fig profiles}
\end{figure}

Only odd frequencies occur in \eqref{eq triangle fourier}, and, divided by $\mu$, the coefficients are the completely multiplicative functions $n\mapsto n^{-2}$ and $n\mapsto\chi_{-4}(n)n^{-2}$ of the odd integers $n$. Hence the Dirichlet series formed with them are Euler products: if $\mathfrak{p} = (p_i)_i$ is the  sequence of prime numbers, we have
\begin{equation}\label{eq euler}
\begin{aligned}
D_{\mathcal C}(s)&=\mu\sum_{n\ \rm odd}n^{-2-s}=\mu\prod_{i=2}^\infty\big(1-p_i^{-2}p_i^{-s}\big)^{-1},\\
D_{\mathcal S}(s)&=\mu\sum_{n\ \rm odd}\chi_{-4}(n)n^{-2-s}=\mu\prod_{i=2}^\infty\big(1-\chi_{-4}(p_i)p_i^{-2}p_i^{-s}\big)^{-1}.
\end{aligned}
\end{equation}
 Their Bohr lifts are $\Theta_{\mathcal C}=\mu P$ and $\Theta_{\mathcal S}=\mu Q$, where
\begin{equation}\label{eq PQ}
P(z)=\prod_{i=2}^\infty(1-p_i^{-2}z_i)^{-1},\qquad Q(z)=\prod_{i=2}^\infty(1-\chi_{-4}(p_i)p_i^{-2}z_i)^{-1}.
\end{equation}

Therefore, the symbol of $\mathfrak R_d=\mathfrak R_d(\sqrt3\mathcal C,\sqrt3\mathcal S)$ is $\Theta=\mu\operatorname{diag}(P,Q)$, whose singular values at $z$ are $\mu|P(z)|$ and $\mu|Q(z)|$, and whose determinant is $\mu^2PQ$. For $z\in\Dinf$ one has $1-p_i^{-2}<|1-p_i^{-2}z_i|<1+p_i^{-2}$ for every $i\geq2$, hence
\[
\frac{\pi^2}{12}=\frac54\,\zeta(4)/\zeta(2)=\prod_{i=2}^\infty\frac1{1+p_i^{-2}}<|P(z)|<\prod_{i=2}^\infty\frac1{1-p_i^{-2}}=\frac34\,\zeta(2)=\frac{\pi^2}{8}.
\]

Both bounds are approached along the points with $z_i=-r$ (resp.\ $z_i=r$) for the first $M$ coordinates and $z_i=0$ otherwise, as $r\to1$ and $M\to\infty$. Since $\chi_{-4}(p)=\pm1$, the same holds for $Q$ (replace $z_i$ by $\chi_{-4}(p_i)z_i$). So $\Theta$ is bounded, $|\det\Theta|\ge\mu^2\pi^4/144>0$, and the extremes over $\Dinf$ of $\smin(\Theta)^2$ and $\smax(\Theta)^2$ are $\mu^2\pi^4/144=2/3$ and $\mu^2\pi^4/64=3/2$ (recall $\mu^2\pi^4=96$). By Theorem~\ref{thm two profiles}, $\mathfrak R_d$ is a Riesz basis with optimal bounds $A=\min\{1,2/3\}=2/3$ and $B=\max\{1,3/2\}=3/2$.
\end{proof}

The case of a single generator, dilated along all directions, is the following.

\begin{corollary}\label{cor general profile}
Let $\psi(t)=\sum_{n\ge1}\big(a_n\sqrt2\cos(2\pi nt)+b_n\sqrt2\sin(2\pi nt)\big)$ be $1$-periodic with mean zero, and let $F_a,F_b$ be the Bohr lifts of $D_a=\sum a_nn^{-s}$ and $D_b=\sum b_nn^{-s}$. Then $$\{1\}\cup\{\psi(\alpha\cdot x):\ \alpha\in\Z^d\setminus\{0\}\}$$ is a Riesz basis of $L^2([0,1]^d)$ if and only if $F_a$ and $F_b$ are bounded and bounded away from zero on $\Dinf$, that is, $D_a,D_b,1/D_a,1/D_b\in\mathcal H_\infty$. The optimal bounds are
\[
A=\min\big\{1,\ 2\inf|F_a|^2,\ 2\inf|F_b|^2\big\},\qquad B=\max\big\{1,\ 2\sup|F_a|^2,\ 2\sup|F_b|^2\big\},
\]
the extremes being taken over $\Dinf$, or over $\Tinf$ when $F_a,F_b$ extend continuously to $\Dcl$.
\end{corollary}

\begin{example}\label{ex single profile}
Let $\psi_\star=\sqrt{3/2}\,(\mathcal C+\mathcal S)$ be a trapezoidal wave (Figure~\ref{fig profiles}), again ReLU-realizable, $1$-periodic and of norm $1$. Its coefficients are $a_n=\mu n^{-2}/\sqrt2$ and $b_n=\mu\chi_{-4}(n)n^{-2}/\sqrt2$ for odd $n$ and $0$ for even $n$, so $2|F_a|^2=|\Theta_{\mathcal C}|^2$ and $2|F_b|^2=|\Theta_{\mathcal S}|^2$. By Corollary~\ref{cor general profile} and the computation in the proof of Theorem~\ref{thm relu}, the dilates of the single function $\psi_\star$, $\{1\}\cup\{\psi_\star(\alpha\cdot x):\ \alpha\in\Z^d\setminus\{0\}\}$, form a Riesz basis of $L^2([0,1]^d)$ with optimal bounds $2/3$ and $3/2$. 
\end{example}

\begin{example}\label{ex nondiagonal}
Let $\mathcal A$ be the $1$-periodic function with
\[
\mathcal A(t)=\sqrt3\,(8t-1)\quad(0\le t\le\tfrac14),\qquad \mathcal A(t)=\frac{\sqrt3}{3}\,(5-8t)\quad(\tfrac14\le t\le1),
\]
an asymmetric triangle wave which rises on a quarter of the period and falls on the remaining three quarters (Figure~\ref{fig profiles}). It is continuous and piecewise linear, hence ReLU-realizable, with the same range $[-\sqrt3,\sqrt3]$ as $\sqrt3\,\mathcal C$, mean zero and $\|\mathcal A\|_{L^2(0,1)}=1$. Its Fourier coefficients with respect to $\sqrt2\cos(2\pi nt)$ and $\sqrt2\sin(2\pi nt)$ are
\[
a_n=-\frac{2\mu}{3}\,\frac{1-\cos(\pi n/2)}{n^2},\qquad b_n=\frac{2\mu}{3}\,\frac{\sin(\pi n/2)}{n^2},
\]
and then

\[
a_n= \begin{cases}
    -\frac{2\mu}{3n^2} & \mbox{ if } n \mbox{ is odd,}\\
    -\frac{4\mu}{3n^2} & \mbox{ if } n \equiv2\pmod4,\\
    0 & \mbox{ if } n \equiv0\pmod4,
\end{cases}
\quad \quad \mbox{and} \quad\quad 
b_n = \begin{cases}
    \frac{2\mu\chi_{-4}(n)}{3n^2} & \mbox{ if } n \mbox{ is odd,}\\
    0 & \mbox{ if } n \mbox{ is even.}
\end{cases}
\]
Regrouping the terms,
\[
D_a(s)=-\frac{2\mu}{3}\big(1+\frac{1}{2}2^{-s}\big)\sum_{n\ \rm odd}\frac{1}{n^2}n^{-s},\qquad D_b(s)=\frac{2\mu}{3}\,\sum\limits_{n=1}^\infty\frac{\chi_{-4}(n)}{n^2} n^{-s},
\]
with Bohr lifts $F_a(z)=-\frac{2\mu}{3}\big(1+\frac{z_1}{2}\big)P(z)$ and $F_b(z)=\frac{2\mu}{3}Q(z)$, where $P$ and $Q$ are the Euler products \eqref{eq PQ}. Let $G=\sum_{n=1}^\infty\chi_{-4}(n) n^{-2}=0.9159\ldots$ be Catalan's constant. The following statements are proved in Section~\ref{subsec proof example}.
\begin{enumerate}[label=(\alph*)]
\item \emph{One generator.} $\{1\}\cup\{\mathcal A(\alpha\cdot x):\alpha\in\Z^d\setminus\{0\}\}$ is a Riesz basis of $L^2([0,1]^d)$ with optimal bounds $A=4/27$ and $B=3$. The asymmetry costs a factor $9/2$ in the lower bound with respect to the symmetric generator $\psi_\star$.
\item \emph{Two generators.} $\{1\}\cup\{\sqrt3\,\mathcal C(\alpha\cdot x),\mathcal A(\alpha\cdot x)\}_{\alpha\succ0}$ is a Riesz basis of $L^2([0,1]^d)$ for every $d$. Its matrix (the \emph{symbol} of Section~\ref{sec proofs})
\[
\Theta_{\sqrt3\mathcal C,\mathcal A}(z)=\mu\begin{pmatrix}P(z)&-\frac23\big(1+\frac{z_1}{2}\big)P(z)\\[2pt] 0&\frac23\,Q(z)\end{pmatrix}
\]
is upper triangular but not diagonal, and its optimal bounds are
\[
B=\Big\|\begin{pmatrix}\sqrt{3/2}&\sqrt{3/2}\\0&\frac23\mu G\end{pmatrix}\Big\|^2=3.1949\ldots,\quad
A=\mu^2\,\Big\|\begin{pmatrix}96G/\pi^4&18/\pi^2\\0&18/\pi^2\end{pmatrix}\Big\|^{-2}=0.1391\ldots,
\]
attained at $w\equiv1$ and at the point $w^-$ of $\Tinf$ with $w^-_1=1$ and $w^-_n=-\chi_{-4}(p_n)$ for $n>1$, respectively.
\end{enumerate}
Compared with the diagonal pair $(\sqrt3\,\mathcal C,\sqrt3\,\mathcal S)$, the coupling term degrades the bounds considerably, although $\mathcal A$ and $\sqrt3\,\mathcal S$ have the same norm and the same range. Note also that at $w^-$ the value $|Q|=\pi^2/12$ is minimal but $|P|=\pi^4/(96G)$ is not: the smallest singular value is more sensitive to the entries involving $Q$ than to $P$ (Lemma~\ref{lem triangular}).
\end{example}

\section{Proofs of the main results}\label{sec proofs}

The proofs of Theorems~\ref{thm relu} and~\ref{thm two profiles} reduce, direction by direction, to the following situation.

\begin{definition}\label{def dilation space}
A \emph{dilation space with two frequency lines} is a Hilbert space $\cH$ with an orthonormal basis $\{v_{1,n},v_{2,n}\}_{\,n\in\N}$. Its \emph{dilations} are the isometries $V_m$, $m\in\N$, determined by $V_mv_{k,n}=v_{k,mn}$, and the closed span of $\{v_{k,n}\}_{n\in\N}$ is the \emph{$k$-th frequency line}.
\end{definition}

Clearly $V_1=I$ and $V_mV_{m'}=V_{mm'}$. The setting of \cite{HeLiSe} is $L^2(0,1)$ with the single line $v_n=\sqrt2\sin(\pi nx)$ and $V_mf=f(m\,\cdot)$; the spaces used in Section~\ref{subsec proof B} have the two lines $\sqrt2\cos(2\pi n\beta\cdot x)$ and $\sqrt2\sin(2\pi n\beta\cdot x)$, for a fixed primitive direction $\beta$.

Given $\varphi_1,\varphi_2\in\cH$, write $\varphi_j=\sum_{k,n}a^{(j)}_{k,n}v_{k,n}$, and let $F_{kj}\in H_2(\Dinf)$ be the Bohr lift of the Dirichlet series $D_{kj}(s)=\sum_na^{(j)}_{k,n}n^{-s}$ formed with the coefficients of $\varphi_j$ along the $k$-th line. The \emph{symbol} of the system
\[
\Phi=\{V_n\varphi_1,V_n\varphi_2\}_{\,n\in\N}
\]
is the $2\times2$ matrix of holomorphic functions $\Theta_\Phi=[F_{kj}]$ on $\Dinf$: its $j$-th column is the Bohr lift of $\varphi_j$, and it is the Bohr lift of the matrix of Dirichlet series $[D_{kj}]$.

\begin{proposition}\label{prop matrix}
Let $\cH$ be a dilation space with two frequency lines, $\varphi_1,\varphi_2\in\cH$, and let $\Theta=\Theta_\Phi$ be the symbol of $\Phi=\{V_n\varphi_1,V_n\varphi_2\}_{\,n\in\N}.$
\begin{enumerate}[label=(\alph*)]
\item\label{matrix bessel} $\Phi$ is a Bessel sequence if and only if $\Theta$ is bounded on $\Dinf$, and then its optimal Bessel bound is $\sup_{z\in\Dinf}\|\Theta(z)\|^2$.
\item\label{matrix riesz} $\Phi$ is a Riesz basis of $\cH$ if and only if $\Theta$ is bounded on $\Dinf$ and $\inf_{z\in\Dinf}|\det\Theta(z)|>0$; equivalently, if and only if $\Theta$ is bounded, $\Theta(z)$ is invertible for every $z\in\Dinf$, and $\sup_z\|\Theta(z)^{-1}\|<\infty$. In that case its optimal Riesz bounds are
\[
A=\inf_{z\in\Dinf}\smin\big(\Theta(z)\big)^2=\Big(\sup_{z\in\Dinf}\|\Theta(z)^{-1}\|\Big)^{-2},\qquad B=\sup_{z\in\Dinf}\smax\big(\Theta(z)\big)^2 .
\]
\end{enumerate}
\end{proposition}
Since the Bohr lift is multiplicative, \ref{matrix riesz} says that $\Phi$ is a Riesz basis if and only if the four $D_{kj}\in\mathcal H_\infty$ and $1/(D_{11}D_{22}-D_{12}D_{21})\in\mathcal H_\infty$. The proof uses in an essential way that the symbol is a square matrix of finite size: a $2\times2$ matrix whose adjoint is bounded below is invertible, and $\Theta^{-1}=(\det\Theta)^{-1}\adj\Theta$. The extension to infinitely many generators and frequency lines is carried out in \cite{ACFS_theory}.
\begin{remark}\label{rem matrix}
 \begin{itemize}
     \item[1)] The same proof, with one generator on a single frequency line and a scalar symbol, gives the theorem of Hedenmalm, Lindqvist and Seip \cite[Theorem~5.2]{HeLiSe}, and \ref{matrix bessel} then says that $\{\varphi(n\,\cdot)\}$ is a Bessel sequence if and only if $D_\varphi$ is a multiplier of the Hardy space $\mathcal H_2$ of Dirichlet series, that is, $D_\varphi\in\mathcal H_\infty$ \cite[Theorem~3.1]{HeLiSe}, \cite{ACS_splitting}.
     \item[2)] The previous scalar case, applied to the cosine and to the sine line, is also the content of the characterization in \cite{BillnerSchneider2026} of the sequences $\tau$ for which the system generated by $\sqrt2\sum\tau_k\cos(2\pi k\,\cdot)$ and $\sqrt2\sum\tau_k\sin(2\pi k\,\cdot)$ is a Riesz basis of the mean-zero subspace of $L^2([0,1]^d)$, namely that $\tau$ and its Dirichlet inverse be multipliers of $\ell_2$ for the Dirichlet convolution. Their sufficient condition, $\tau\in\ell_1$ and $\inf_{\Re s\ge0}|\sum\tau_nn^{-s}|>0$, implies $D,1/D\in\mathcal H_\infty$, which requires neither summable coefficients nor a bound away from zero on the closed half-plane.
 \end{itemize} 
\end{remark}

For the proof of the Proposition \ref{prop matrix}, we need to introduce some definitions. The Hardy space $H_2(\Dinf,\C^2)$ is the Hilbert space of functions $G:\Dinf\to\C^2$ of the form $G(z)=\sum_{\alpha\in\mathbb{N}_0^{(\mathbb{N})}}c_\alpha z^\alpha$, $c_\alpha\in\C^2$, with $\|G\|^2=\sum_\alpha\|c_\alpha\|^2<\infty$ (the vector-valued version of $H_2(\Dinf)$). Note that for $z$ in $\Dinf$ we have
\begin{equation}\label{eq kernel norm}
\sum_{\alpha\in\mathbb{N}_0^{(\mathbb{N})}}|z^\alpha|^2=\prod_{i=1}^\infty\frac1{1-|z_i|^2}<\infty
\end{equation} and this expression is locally bounded in $\Dinf$.
Therefore, by Cauchy--Schwarz the series defining $G$ converges  absolutely and locally uniformly on $\Dinf$.
 Hence $G$ is holomorphic, and the $c_\alpha$ are the Taylor coefficients of its restrictions to finite-dimensional slices, so they are determined by $G$. If $b_1,b_2$ is the canonical basis of $\C^2$, then $\{b_1z^\alpha,b_2z^\alpha\}_{\alpha\in\mathbb{N}_0^{(\mathbb{N})}}$ is an orthonormal basis of $\HN$ with the natural inner product, and the $\C^2$-valued polynomials (finite linear combinations of the $b_kz^\alpha$) are dense (see \cite[Chapter 13]{DirSer} for details).

Note that  $\HN$ can be seen as a Hilbert space with reproducing kernel in the following sense. For $z\in\Dinf$, let 
$$k_z(\zeta)=\sum_{\alpha\in\mathbb{N}_0^{(\mathbb{N})}}\overline z^{\,\alpha}\zeta^\alpha=\prod_{i=1}^\infty\frac1{1-\zeta_i\overline{z_i}}.$$
By \eqref{eq kernel norm}, $k_z\in H_2(\Dinf)$ with $\|k_z\|^2=k_z(z)$, and for $G=\sum c_\alpha z^\alpha\in\HN$ and $u\in\C^2$,
\begin{equation}\label{eq reproducing}
\pint{G,\,u\,k_z}=\sum_\alpha\pint{c_\alpha,u}z^\alpha=\pint{G(z),u}_{\C^2}.
\end{equation}
Consequently $\|G(z)\|\le\|G\|\,\|k_z\|$, so point evaluations are continuous and convergence in $\HN$ implies pointwise convergence. 

Let $H_\infty(\Dinf,\MN)$ be the space of bounded holomorphic functions $\Theta:\Dinf\to\MN$, with $\|\Theta\|_\infty=\sup_z\|\Theta(z)\|$. The following  lemma is a particular case of  the easy implication in  \cite[Theorem~3.1]{ACFS_multipliers}.

\begin{lemma}\label{lem easy multiplier}
Let $\Theta\in H_\infty(\Dinf,\MN)$. Then $G\mapsto\Theta G$ is a bounded operator $M_\Theta$ on $\HN$, with $\|M_\Theta\|\le\|\Theta\|_\infty$.
\end{lemma}

Let $\cH$ be a dilation space with two frequency lines. The \emph{Bohr lift} of $\cH$ is the unitary operator
\[
\cB:\cH\to\HN,\qquad \cB(v_{k,n})=b_k\,z^{\alpha(n)}.
\]
Since $\alpha(mn)=\alpha(m)+\alpha(n)$,
\begin{equation}\label{eq intertwining}
\cB\,V_n=M_{z^{\alpha(n)}}\,\cB\qquad(n\in\N),
\end{equation}
that is, the dilation by $n$ becomes multiplication by the monomial $z^{\alpha(n)}$, the Bohr lift of $n^{-s}$. For $\varphi_j=\sum_{k,n}a^{(j)}_{k,n}v_{k,n}$ we get $\cB\varphi_j=\sum_k\big(\sum_na^{(j)}_{k,n}z^{\alpha(n)}\big)b_k=\sum_kF_{kj}b_k=\Theta_\Phi\,b_j$, the $j$-th column of the symbol, and by \eqref{eq intertwining}
\begin{equation}\label{eq synthesis}
\cB(V_n\varphi_j)=z^{\alpha(n)}\,\Theta_\Phi\,b_j=\Theta_\Phi\cdot\big(b_jz^{\alpha(n)}\big).
\end{equation}
The \emph{synthesis operator} of a family $\{f_i\}_{i\in I}$ is the map $e_i\mapsto f_i$ on the finitely supported sequences of $\ell^2(I)$. The family is a Bessel sequence with bound $B$ if and only if this map extends to a bounded operator with norm at most $\sqrt B$ \cite[Theorem~3.2.3]{Christensen2016}. Let $\cT$ be the synthesis operator of $\Phi$ transported by $\cB$, defined on the polynomials of $\HN$ by $\cT(b_jz^{\alpha(n)})=\cB(V_n\varphi_j)$, that is, by \eqref{eq synthesis},
\begin{equation}\label{eq T mult}
\cT G=\Theta_\Phi\,G\qquad\text{for every $\C^2$-valued polynomial }G .
\end{equation}
Since $\{b_jz^{\alpha(n)}\}$ is an orthonormal basis of $\HN$ and $\cB$ is unitary, Section~\ref{subsec notation} gives: $\Phi$ is a Bessel sequence with bound $B$ if and only if $\cT$ extends to a bounded operator on $\HN$ with $\|\cT\|^2\le B$, and $\Phi$ is a Riesz basis of $\cH$ with bounds $A,B$ if and only if $\cT$ extends to a bounded invertible operator on $\HN$ with $\|\cT\|^2\le B$ and $\|\cT^{-1}\|^{-2}\ge A$.

We can now prove Proposition \ref{prop matrix}.
\begin{proof}[Proof of Proposition~\ref{prop matrix}]
Write $\Theta=\Theta_\Phi$ and let $\cT$ be as in \eqref{eq T mult}. If $\cT$ extends to a bounded operator on $\HN$, then $$(\cT G)(z)=\Theta(z)G(z)$$ for all $G\in\HN$. For polynomials this identity is \eqref{eq T mult}. For $G\in\HN$ take polynomials $G_m\to G$ in $\HN$; then $\cT G_m\to\cT G$ in $\HN$, hence pointwise, while $\Theta(z)G_m(z)\to\Theta(z)G(z)$.

If $z\in\Dinf$ and $u\in \mathbb{C}^2$, then
\begin{equation}\label{eq adjoint kernel}
\cT^*(u\,k_z)=\Theta(z)^*u\;k_z.
\end{equation}
In fact, by \eqref{eq reproducing}, for every $G\in\HN$,
\[
\pint{G,\cT^*(uk_z)}=\pint{\cT G,uk_z}=\pint{\Theta(z)G(z),u}=\pint{G(z),\Theta(z)^*u}=\pint{G,\Theta(z)^*u\,k_z}.
\]
Taking norms in \eqref{eq adjoint kernel} and dividing by $\|k_z\|$ gives $\|\Theta(z)^*u\|=\|\cT^*(uk_z)\|/\|k_z\|$, which is at most $\|\cT^*\|\|u\|=\|\cT\|\|u\|$, and at least $\|(\cT^*)^{-1}\|^{-1}\|u\|=\|\cT^{-1}\|^{-1}\|u\|$ when $\cT$ is invertible.

\ref{matrix bessel} If $\Theta$ is bounded, Lemma~\ref{lem easy multiplier} and \eqref{eq T mult} show that $\cT$ extends to $M_\Theta$, so $\Phi$ is a Bessel sequence with bound $\|\Theta\|_\infty^2$. Conversely, if $\cT$ extends to a bounded operator, then we have that $\|\Theta(z)\|=\|\Theta(z)^*\|\le\|\cT\|$ for every $z$. Hence $\|\cT\|=\|\Theta\|_\infty$, and the optimal Bessel bound is $\|\cT\|^2=\|\Theta\|_\infty^2$.

\ref{matrix riesz} Let us begin by seeing that the two conditions are equivalent. Let $s_1(z)\ge s_2(z)$ be the singular values of $\Theta(z)$, so that $|\det\Theta(z)|=s_1(z)s_2(z)$, $s_1=\smax$, $s_2=\smin$, and $\|\Theta(z)^{-1}\|=s_2(z)^{-1}=s_1(z)/|\det\Theta(z)|$ when $\Theta(z)$ is invertible. If $\Theta$ is bounded and $\delta=\inf_z|\det\Theta(z)|>0$, then every $\Theta(z)$ is invertible and $\|\Theta(z)^{-1}\|\le\|\Theta\|_\infty/\delta$; conversely, if $\sup_z\|\Theta(z)^{-1}\|=c<\infty$, then $|\det\Theta(z)|\ge s_2(z)^2\ge c^{-2}$ for every $z$.

Suppose now that $\Theta$ is bounded and $\sup_z\|\Theta(z)^{-1}\|<\infty$. Given the equality $$\Theta^{-1}=(\det\Theta)^{-1}\adj\Theta,$$
then $\Theta^{-1}$ is holomorphic on $\Dinf$, because $\det\Theta$ is holomorphic, does not vanish, and it is bounded. Hence $M_\Theta$ and $M_{\Theta^{-1}}$ are bounded operators on $\HN$, by Lemma~\ref{lem easy multiplier}, and  inverse to each other. By \eqref{eq T mult}, $\cT$ extends to the invertible operator $M_\Theta$, so $\Phi$ is a Riesz basis, with $\|\cT^{-1}\|=\|M_{\Theta^{-1}}\|\le\sup_z\|\Theta(z)^{-1}\|$.

Suppose conversely that $\Phi$ is a Riesz basis, so that $\cT$ extends to a bounded invertible operator. By \ref{matrix bessel}, $\Theta$ is bounded. Then, $\|\Theta(z)^*u\|\ge\|\cT^{-1}\|^{-1}\|u\|$ for all $u\in\C^2$, so the $2\times2$ matrix $\Theta(z)^*$ is injective, hence invertible, with $\smin(\Theta(z))=\smin(\Theta(z)^*)\ge\|\cT^{-1}\|^{-1}$. Thus $\Theta(z)$ is invertible and $\|\Theta(z)^{-1}\|\le\|\cT^{-1}\|$ for every $z$.

The two paragraphs together show that $\|\cT^{-1}\|=\sup_z\|\Theta(z)^{-1}\|=\big(\inf_z\smin(\Theta(z))\big)^{-1}$ whenever $\Phi$ is a Riesz basis, and \ref{matrix bessel} shows that $\|\cT\|=\sup_z\smax(\Theta(z))$. Since the optimal Riesz bounds are $\|\cT^{-1}\|^{-2}$ and $\|\cT\|^2$, the proof is complete.
\end{proof}

Proposition~\ref{prop matrix} identifies the optimal Riesz bounds with the extremes of the singular values of $\Theta$ over $\Dinf$. Computing these extremes directly can be difficult, since $\Dinf$ is not compact and they need not be attained. When the entries of $\Theta$ extend continuously to the closed polydisk, the following lemma reduces the problem to the polytorus $\Tinf$, where the extremes are attained.

\begin{lemma}\label{lem torus}
Let $\Theta:\Dcl\to\MN$ be continuous for the product topology and holomorphic on $\Dinf$. Then $\sup_{z\in\Dinf}\|\Theta(z)\|=\max_{w\in\Tinf}\|\Theta(w)\|$. If moreover $\det\Theta$ has no zeros on $\Dcl$, then $\inf_{z\in\Dinf}\smin(\Theta(z))=\min_{w\in\Tinf}\smin(\Theta(w))$.
\end{lemma}

\begin{proof}
Let $F:\Dcl\to\C$ be continuous and holomorphic on $\Dinf$; we show $\sup_{\Dinf}|F|=\max_{\Tinf}|F|$. By the density of $\Dinf$ in $\Dcl$ and continuity $\sup_{\Dinf}|F|\geq\max_{\Tinf}|F|$. For the other inequality, fix $z\in\Dinf$ and $M\in\N$. The function $\zeta\mapsto F(\zeta_1,\dots,\zeta_M,z_{{M+1}},z_{{M+2}},\dots)$ is continuous on $\overline\D^M$ and holomorphic on $\D^M$, so by the maximum principle there is $\eta^{(M)}\in\T^M$ with $|F(z)|\le|F(\eta^{(M)},z_{{M+1}},\dots)|$. The points $(\eta^{(M)},z_{{M+1}},\dots)$ have a subsequence converging in the compact space $\Dcl$ to a point $w$, and $|w_i|=1$ for every $i$, so $w\in\Tinf$ and $|F(z)|\le|F(w)|$. Apply this to $F=\pint{\Theta(\cdot)x,u}$, for unit vectors $x,u\in\C^2$, to get $\|\Theta(z)\|\le\sup_{\Tinf}\|\Theta\|$, and to $F=\pint{\Theta(\cdot)^{-1}x,u}$, the entries of $\Theta^{-1}=(\det\Theta)^{-1}\adj\Theta$ being continuous on $\Dcl$ and holomorphic on $\Dinf$, to get the statement for $\smin(\Theta)=\|\Theta^{-1}\|^{-1}$. The maxima and minima over the compact set $\Tinf$ are attained since $\|\Theta\|$ and $\smin(\Theta)$ are continuous.
\end{proof}

\subsection{Proof of Theorem~\ref{thm two profiles} and Corollary~\ref{cor general profile}}\label{subsec proof B}

For a primitive $\beta\in\Z^d$ with $\beta\succ0$, let
\[
\cH^c_\beta=\spanc\{v^c_{\beta,n}=\sqrt2\cos(2\pi n\beta\cdot x)\}_{n\in\N},\qquad \cH^s_\beta=\spanc\{v^s_{\beta,n}=\sqrt2\sin(2\pi n\beta\cdot x)\}_{n\in\N},
\]
and $\cH_\beta=\cH^c_\beta\oplus\cH^s_\beta$. Since every $\alpha\succ0$ is uniquely $\alpha=n\beta$ with $n\in\N$ and $\beta\succ0$ primitive, the real trigonometric system gives the orthogonal decomposition
\begin{equation}\label{eq decomposition directions}
L^2([0,1]^d)=\C\oplus\bigoplus_{\beta\succ0\ \text{primitive}}\big(\cH^c_\beta\oplus\cH^s_\beta\big).
\end{equation}
The dilation $V_mf=f(m\,\cdot)$, functions being extended periodically, is an isometry of $L^2([0,1]^d)$ mapping $v^\ast_{\beta,n}$ to $v^\ast_{\beta,mn}$ ($\ast=c,s$), so each $\cH_\beta$ is a dilation space with two frequency lines in the sense of Definition~\ref{def dilation space}.

\begin{proof}[Proof of Theorem~\ref{thm two profiles}]
Fix a primitive $\beta\succ0$ and let $\Phi_\beta=\{\psi_1(n\beta\cdot x),\psi_2(n\beta\cdot x)\}_{n\in\N}$.

Let $\lambda_d$ be the Lebesgue measure on $[0,1]^d$, identifying $\mathbb{T} = \mathbb{R}/\mathbb{Z} = [0,1)$ and taking $\pi : [0,1]^d \to \mathbb{T}$ defined by $\pi(x) = \beta\cdot x -  \lfloor \beta\cdot x \rfloor$. Then, the measure $\mu(E) = \lambda_d(\pi^{-1}(E))$ satisfies that for every $1-$periodic function $g$
$$\int\limits_{[0,1]^d} g(\beta\cdot x) dx = \int\limits_{\mathbb{T}} g(t)d\mu(t),$$
and then its Fourier coefficients are
$$\hat{\mu}(k) = \int\limits_{[0,1]^d} e^{-2\pi i k \beta\cdot x} dx = \prod\limits_{i=1}^d \int\limits_0^1e^{-2\pi i k \beta_i x_i} dx_i =\begin{cases}
    0 & \mbox{ if } k\not=0,\\
    1 & \mbox{ if } k=0.
 \end{cases}$$
By the uniqueness of the Fourier coefficients, $\mu$ is the Lebesgue measure. Since

\[
\int_{[0,1]^d}|f(\beta\cdot x)|^2\,dx=\int_{\mathbb T}\vert f(t) \vert^2\,d\mu=\int_0^1|f(t)|^2\,dt,
\]
then the map $f\mapsto f(\beta\cdot x)$ is an isometry of $L^2(0,1)$ into $L^2([0,1]^d)$. So the Fourier series of $\psi_i$ gives $\psi_i(\beta\cdot x)=\sum_na^{(i)}_nv^c_{\beta,n}+b^{(i)}_nv^s_{\beta,n}\in\cH_\beta$. Since $\psi_i(n\beta\cdot x)=V_n\psi_i(\beta\cdot x)$, $\Phi_\beta$ is the system of dilations of the two generators $\psi_1(\beta\cdot x),\psi_2(\beta\cdot x)$ of $\cH_\beta$, and the Bohr lift of $\psi_i(\beta\cdot x)$ is the column $(F_{a^{(i)}},F_{b^{(i)}})^{T}$, so the symbol of $\Phi_\beta$ is $\Theta_{\psi_1,\psi_2}$ for every $\beta$. By Proposition~\ref{prop matrix}\ref{matrix riesz}, $\Phi_\beta$ is a Riesz basis of $\cH_\beta$ if and only if the four entries of $\Theta_{\psi_1,\psi_2}$ are bounded and $\inf_z|\det\Theta_{\psi_1,\psi_2}(z)|>0$, and then its optimal bounds are $A_\Theta=\inf_z\smin(\Theta_{\psi_1,\psi_2}(z))^2$ and $B_\Theta=\sup_z\smax(\Theta_{\psi_1,\psi_2}(z))^2$.

By the uniqueness of the decomposition $\alpha=n\beta$, $\mathfrak R_d(\psi_1,\psi_2)$ is the disjoint union of $\{1\}$ and the families $\Phi_\beta$, which lie in the mutually orthogonal subspaces of \eqref{eq decomposition directions}. For a finite linear combination $f=c_0+\sum_\beta f_\beta$, with $f_\beta$ a combination of elements of $\Phi_\beta$, we have $\|f\|^2=|c_0|^2+\sum_\beta\|f_\beta\|^2$, and the coefficient norms split in the same way. Hence the Riesz inequalities hold for $\mathfrak R_d(\psi_1,\psi_2)$ with constants $A,B$ if and only if they hold with the same constants for $\{1\}$ and for every $\Phi_\beta$. Moreover $\mathfrak R_d(\psi_1,\psi_2)$ is complete if and only if each $\Phi_\beta$ is complete in $\cH_\beta$, since an element of $\cH_\beta$ orthogonal to $\Phi_\beta$ is orthogonal to the whole union, and if each $\Phi_\beta$ spans $\cH_\beta$ the union spans all the summands of \eqref{eq decomposition directions}. Therefore $\mathfrak R_d(\psi_1,\psi_2)$ is a Riesz basis if and only if every $\Phi_\beta$ is a Riesz basis of $\cH_\beta$. The bounds needed for the ``if'' direction are uniform in $\beta$ because all the $\Phi_\beta$ have the same symbol, hence the same optimal bounds $A_\Theta,B_\Theta$, and the optimal bounds of $\mathfrak R_d(\psi_1,\psi_2)$ are $\min\{1,A_\Theta\}$ and $\max\{1,B_\Theta\}$.

By Section~\ref{subsec notation}, an entry of $\Theta_{\psi_1,\psi_2}$ is bounded on $\Dinf$ if and only if the corresponding Dirichlet series belongs to $\mathcal H_\infty$, and the determinant, which is the Bohr lift of $D_{a^{(1)}}D_{b^{(2)}}-D_{a^{(2)}}D_{b^{(1)}}$ and is bounded when the entries are, is bounded away from $0$ if and only if the reciprocal of that series belongs to $\mathcal H_\infty$. 
\end{proof}

\begin{proof}[Proof of Corollary~\ref{cor general profile}]
Apply Theorem~\ref{thm two profiles} with $\psi_1=\psi$ and $\psi_2=\psi(-\,\cdot)$, whose coefficients are $(a,-b)$. Then $\{\psi_i(\alpha\cdot x):\alpha\succ0\}=\{\psi(\alpha\cdot x):\alpha\neq0\}$. We have
\[
\Theta_{\psi_1,\psi_2}=\begin{pmatrix}F_a&F_a\\F_b&-F_b\end{pmatrix}=\begin{pmatrix}F_a&0\\0&F_b\end{pmatrix}\begin{pmatrix}1&1\\1&-1\end{pmatrix},
\]
the last factor being $\sqrt2$ times a unitary. So the singular values are $\sqrt2|F_a|$ and $\sqrt2|F_b|$, and the determinant is $-2F_aF_b$, which is bounded away from zero if and only if both $F_a$ and $F_b$ are.
\end{proof}

\subsection{Proof of the claims of Example~\ref{ex nondiagonal}}\label{subsec proof example}

Let us remember that the $n$-th Fourier coefficient of $\mathcal A$ is \[
a_n = \begin{cases}
    \frac{-2\mu}{3n^2} & \mbox{ if } n \mbox{ is odd,}\\
    \frac{-4\mu}{3n^2} & \mbox{ if } n \equiv2\pmod4,\\
    0 & \mbox{ if } n \equiv0\pmod4.
\end{cases}
\quad \quad \mbox{and} \quad\quad 
b_n= \begin{cases}
    \frac{2\mu\chi_{-4}(n)}{3n^2} & \mbox{ if } n \mbox{ is odd,}\\
    0 & \mbox{ if } n \mbox{ is even,}
\end{cases}
\]

And then,
\[
D_a(s)=-\frac{2\mu}{3}\big(1+\frac{1}{2}2^{-s}\big)\sum_{n\ \rm odd}\frac{1}{n^2}n^{-s},\qquad D_b(s)=\frac{2\mu}{3}\,\sum\limits_{n=1}^\infty\frac{\chi_{-4}(n)}{n^2} n^{-s},
\]
with Bohr lifts $F_a(z)=-\frac{2\mu}{3}\big(1+\frac{z_1}{2}\big)P(z)$ and $F_b(z)=\frac{2\mu}{3}Q(z)$, where $P$ and $Q$ are the Euler products \eqref{eq PQ}. Since $\sum_{i=1}^\infty p_i^{-2}<\infty$ and $|1-p_i^{-2}z_i|\ge1-p_i^{-2}$, both products converge uniformly on $\Dcl$, so $P$ and $Q$ extend continuously to $\Dcl$ and have no zeros there.

\emph{(a) One generator.} On $\Dcl$, $\frac12\le|1+\frac{z_1}2|\le\frac32$ and $\frac{\pi^2}{12}\le|P|,|Q|\le\frac{\pi^2}{8}$, and all these extremes are attained on $\Tinf$ (at $z_1=\mp1$, $z_i=\mp1$, resp.\ $z_i=\mp\chi_{-4}(p_i)$). Hence $\min|F_a|=\frac{2\mu}{3}\cdot\frac12\cdot\frac{\pi^2}{12}=\frac{\mu\pi^2}{36}$, $\max|F_a|=\frac{\mu\pi^2}{8}$, $\min|F_b|=\frac{\mu\pi^2}{18}$ and $\max|F_b|=\frac{\mu\pi^2}{12}$, all positive, so $F_a,F_b$ are bounded and bounded away from zero, and Corollary~\ref{cor general profile} gives, using $\mu^2\pi^4=96$,
\[
A=\min\{1,\ 2\min|F_a|^2,\ 2\min|F_b|^2\}=\min\Big\{1,\ \frac4{27},\ \frac{16}{27}\Big\}=\frac4{27},\qquad
B=\max\{1,\ 3,\ \tfrac43\}=3 .
\]

\emph{(b) Two generators.} The symbol is as displayed in Example~\ref{ex nondiagonal}(b): its columns are the Bohr lifts $(\Theta_{\mathcal C},0)^T$ and $(F_a,F_b)^T$. Its entries are bounded ($|1+\frac{z_1}2|\le\frac32$ and $|P|,|Q|\le\pi^2/8$) and its determinant $\frac{2\mu^2}{3}PQ$ is bounded away from zero on $\Dinf$, so by Theorem~\ref{thm two profiles} the system is a Riesz basis for every $d$, and since the entries extend continuously to $\Dcl$ its optimal bounds are $A=\min\{1,\min_{\Tinf}\smin(\Theta)^2\}$ and $B=\max\{1,\max_{\Tinf}\|\Theta\|^2\}$. We need a lemma on triangular $2\times2$ matrices.

\begin{lemma}\label{lem triangular}
For $a,b,c\ge0$ let $\nu(a,b,c)=\big\|\left(\begin{smallmatrix}a&b\\0&c\end{smallmatrix}\right)\big\|^2$.
\begin{enumerate}[label=(\roman*)]
\item The singular values of $\left(\begin{smallmatrix}\xi&\eta\\0&\zeta\end{smallmatrix}\right)$, $\xi,\eta,\zeta\in\C$, are those of $\left(\begin{smallmatrix}|\xi|&|\eta|\\0&|\zeta|\end{smallmatrix}\right)$, and $\nu$ is nondecreasing in each of its arguments.
\item Put $R(q)=\dfrac{2q\,\big(\sqrt{1+4q^2}+2q\big)}{\sqrt{1+4q^2}+1}$ for $q>0$, then $R(q)\ge2q$ for $q\ge\frac12$. For $a,b>0$,
\[
\frac{a\,\partial_a\nu(a,a,b)}{b\,\partial_b\nu(a,a,b)}=R\Big(\frac{a^2}{b^2}\Big),\qquad
\frac{b\,\partial_b\nu(a,b,b)}{a\,\partial_a\nu(a,b,b)}=R\Big(\frac{b^2}{a^2}\Big),
\]
where $\nu(a,a,b)$ and $\nu(a,b,b)$ are regarded as functions of the two variables $(a,b)$, and all four partial derivatives are positive.
\end{enumerate}
\end{lemma}

\begin{proof}
(i) The singular values are the square roots of the eigenvalues of $X^*X$, whose characteristic polynomial is $\lambda^2-(|\xi|^2+|\eta|^2+|\zeta|^2)\lambda+|\xi\zeta|^2$. For a matrix with nonnegative entries the norm is attained at a vector with nonnegative entries, and $\|Xv\|$ is then nondecreasing in each entry of $X$.

(ii) Let $S=a^2+b^2+c^2$ and $\Delta=\sqrt{S^2-4a^2c^2}$. Then we have $\nu=\sigma^2_{\mbox{max}}=(S+\Delta)/2$.

For $\nu(a,a,b)$: $S=2a^2+b^2$ and $\Delta=\sqrt{4a^4+b^4}$. Differentiating, $a\partial_a\nu=2a^2\big(1+2a^2/\Delta\big)$ and $b\partial_b\nu=b^2\big(1+b^2/\Delta\big)$, with $q=a^2/b^2$, $\Delta=b^2\sqrt{1+4q^2}$ and the quotient is $R(q)$.

For $\nu(a,b,b)$: $S=a^2+2b^2$ and $\Delta=\sqrt{a^4+4b^4}$. Therefore $a\partial_a\nu=a^2(1+a^2/\Delta)$, and again 
$b\partial_b\nu=2b^2(1+2b^2/\Delta)$, being the quotient $R(b^2/a^2)$. 

Finally $R(q)\ge2q$ if and only if $\sqrt{1+4q^2}+2q\ge\sqrt{1+4q^2}+1$, that is, $q\ge\frac12$.
\end{proof}

We now maximize $\|\Theta\|^2$ and $\|\Theta^{-1}\|^2$ over $\Tinf$. Write 
$$X=|P(w)| = \prod\limits_{i=2}^\infty \vert (1-p_i^{-2}w_i)^{-1}\vert, \quad\quad Y=|Q(w)| = \prod\limits_{i=2}^\infty \vert (1-\chi_{-4}(p_i)p_i^{-2}w_i)^{-1}\vert,$$ and $u=1+w_1/2$. Since $\Theta^{-1}=\mu^{-1}\left(\begin{smallmatrix}P^{-1}&uQ^{-1}\\0&\frac32Q^{-1}\end{smallmatrix}\right)$, Lemma~\ref{lem triangular}(i) gives
\[
\|\Theta(w)\|^2=\mu^2\,\nu\big(X,\tfrac23|u|X,\tfrac23Y\big),\qquad
\|\Theta(w)^{-1}\|^2=\mu^{-2}\,\nu\big(X^{-1},|u|Y^{-1},\tfrac32Y^{-1}\big).
\]
Both are nondecreasing in $|u|$, and $X$, $Y$ do not depend on $w_1$, while $|u|\le\frac32$ with equality if and only if $w_1=1$. So both maxima are attained with $w_1=1$, where the two expressions become $\mu^2\,\nu(X,X,\frac23Y)$ and $\mu^{-2}\,\nu(X^{-1},\frac32Y^{-1},\frac32Y^{-1})$, and only $(w_i)_{i>1}$ remains. On $\Tinf$ each factor depends only on $c_i=\Re w_i\in[-1,1]$:
\[
X=\prod_{i=2}^\infty\big(1+p_i^{-4}-2p_i^{-2}c_i\big)^{-1/2},\qquad Y=\prod_{i=2}^\infty\big(1+p_i^{-4}-2\chi_{-4}(p_i)p_i^{-2}c_i\big)^{-1/2}.
\]
Put $x=\log X$, $y=\log Y$, and regard
\[
\Psi(X,Y)=\nu(X,X,\tfrac23Y)\qquad\text{and}\qquad\Omega(X,Y)=\nu(X^{-1},\tfrac32Y^{-1},\tfrac32Y^{-1})
\]
as functions of $(x,y)$, we see that $\Psi$ is increasing in $x$ and in $y$, that $\Omega$ is decreasing in $x$ and in $y$. Applying Lemma~\ref{lem triangular}(ii) with $a=e^x$, and $b= \frac23e^y$
\[
\frac{\partial_x\Psi(e^x,e^y)}{\partial_y\Psi(e^x,e^y)}= \frac{a\partial_{a}\nu(a,a,b)}{b\partial_{b}\nu(a,a,b)}=R\Big(\frac{a^2}{b^2 }\Big) = R\Big(\frac{9X^2}{4Y^2}\Big),
\]
and with $a=e^{-x}$, $b=\frac32e^{-y}$
\[\frac{\partial_y\Omega(e^x,e^y)}{\partial_x\Omega(e^x,e^y)} =  \frac{b\partial_b\nu(a,b,b)}{a\partial_a\nu(a,b,b)}=R\Big(\frac{b^2}{a^2 }\Big) = R\Big(\frac{9X^2}{4Y^2}\Big).
\]

Since $X\ge\pi^2/12$ and $Y\le\pi^2/8$ on $\Tinf$, $\frac{9X^2}{4Y^2}\ge\frac{9(\pi^2/12)^2}{4(\pi^2/8)^2}=1$, so $R(\frac{9X^2}{4Y^2})\ge2\frac{9X^2}{4Y^2}\ge2$.

For each $i\geq 2$
\[
\frac{\partial x}{\partial c_{i}}=\frac{p_i^{-2}}{1+p_i^{-4}-2p_i^{-2}c_i},
\]
and
$$
\frac{\partial y}{\partial c_{i}}=\begin{cases}
    \frac{p_i^{-2}}{1+p_i^{-4}-2p_i^{-2}c_i}, &\mbox{ if } p_i\equiv1\pmod4,\\
    \frac{-p_i^{-2}}{1+p_i^{-4}+2p_i^{-2}c_i}, &\mbox{ if } p_i\equiv3\pmod4.
\end{cases}
$$
Therefore, for $p_i\equiv1\pmod4$ 
$\Psi$ is maximal at $c_i=1$ and $\Omega$ at $c_i=-1$. For $p_i\equiv3\pmod4$, $x$ is increasing in $c_i$, and $y$ is decreasing, with
\[
\frac{16}{25}\le\Big(\frac{1-p_i^{-2}}{1+p_i^{-2}}\Big)^2\le\frac{\partial x/\partial c_{i}}{|\partial y/\partial c_i|}\le\Big(\frac{1+p_i^{-2}}{1-p_i^{-2}}\Big)^2\le\frac{25}{16},
\]
the outer bounds being the values for $p_2=3$. Therefore
\[
\frac{\partial\Psi}{\partial c_i}=\partial_y\Psi\Big[R\Big(\frac{9X^2}{4Y^2}\Big)\frac{\partial x}{\partial c_i}+\frac{\partial y}{\partial c_i}\Big]>0,\qquad
\frac{\partial\Omega}{\partial c_i}=|\partial_x\Omega|\Big[R\Big(\frac{9X^2}{4Y^2}\Big)\Big|\frac{\partial y}{\partial c_i}\Big|-\frac{\partial x}{\partial c_i}\Big]>0,
\]
because in both brackets the ratio of the two rates is at most $25/16<2\le R(\frac{9X^2}{4Y^2})$. So $\Psi$ and $\Omega$ are increasing in $c_i$, and both are maximal at $c_i=1$. Changing the coordinates of an arbitrary $w\in\Tinf$ one at a time to these optimal values does not decrease $\Psi$ or $\Omega$, and the modified points converge in $\Tinf$ to the optimal configuration.

Then, by continuity, the maximum of $\Psi$ over $\Tinf$ is attained at $w\equiv1$, that is, at $(X,Y)=(\pi^2/8,G)$, where $G=\sum_{n=1}^\infty\chi_{-4}(n)\,n^{-2}=0.9159\ldots$ is again Catalan's constant. The maximum of $\Omega$ is attained at $w^-$, the sequence given by $w^-_1=1$ and $w^-_i=-\chi_{-4}(p_i)$ for $i\geq2$. At this point we have
\[
X=\prod_{p_i\equiv1(4)}\frac1{1+p_i^{-2}}\prod_{p_i\equiv3(4)}\frac1{1-p_i^{-2}}=\frac{\pi^4}{96\,G},\qquad Y=\prod_{i=2}^\infty\frac1{1+p_i^{-2}}=\frac{\pi^2}{12}.
\]
To see the value of $X$, multiply the product defining it by
\[
\prod_{p_i\equiv1(4)}\frac1{1-p_i^{-2}}\prod_{p_i\equiv3(4)}\frac1{1+p_i^{-2}}=G.
\]
Each prime then contributes a factor $(1-p_i^{-4})^{-1}$, so the result is $\prod_{i=2}^\infty(1-p_i^{-4})^{-1}=\frac{\pi^2}{8}\cdot\frac{\pi^2}{12}=\frac{\pi^4}{96}$, and hence $X=\pi^4/(96G)$.

Finally, at $w\equiv1$ we have $\mu\pi^2/8=\sqrt{3/2}$, while at $w^-$ we have $X^{-1}=96G/\pi^4$ and $\frac32Y^{-1}=18/\pi^2$. These are precisely the values of $B$ and $A$ stated in Example~\ref{ex nondiagonal}(b); numerically, $B=3.19495\ldots$ and $A=0.13910\ldots$, so that $\min\{1,A\}=A$ and $\max\{1,B\}=B$. \qed

\begin{remark}
For the triangle waves and for the pair of Example~\ref{ex nondiagonal} the symbol has Euler-product entries, which is what makes the constants explicit. For other piecewise linear generators the Dirichlet series will in general not be Euler products, the invertibility of the determinant must be tested on $\Dinf$ rather than on $\Tinf$ (the function $z_1$ is unimodular on $\Tinf$ but vanishes at $0$; in terms of dilations, $D_\varphi(s)=2^{-s}$ gives an orthonormal sequence which is not a basis \cite[Example~4.2]{ACS_splitting}), and the extremes will be harder to compute.

If three or more functions are dilated along all directions $\alpha\succ0$, the system cannot be a Riesz basis but it may be a frame (see \cite{ACFS_theory} for details).
\end{remark}

\section*{Declaration on the use of AI tools}
The main results of this paper were obtained before any use of AI tools. The authors used Claude (Anthropic), mainly the Opus 5.5 model, to assist in restructuring parts of the exposition, working out some of the examples, and improving the readability and language of the manuscript. The authors reviewed all AI-assisted content and take full responsibility for the entire manuscript.

\bibliographystyle{abbrv}
\bibliography{refs_bohr}

\bigskip
\begin{flushleft}
\noindent
J.~Antezana\\
Departament de Matem\`atiques i Inform\`atica, Universitat de Barcelona, Spain, and Centre de Recerca Matem\`atica, Barcelona, Spain.
\\ \texttt{jorge.antezana@ub.edu}

\medskip
\noindent
D.~Carando, T.~Fern\'andez Vidal, M. Scotti\\
Departamento de Matem\'atica, FCEN, Universidad de Buenos Aires, and IMAS (UBA--CONICET), Buenos Aires, Argentina.\\ \texttt{dcarando@dm.uba.ar}, \texttt{tfvidal@dm.uba.ar}, \texttt{mscotti@dm.uba.ar}.

\end{flushleft}

\end{document}